\documentclass[11pt,a4paper,twoside]{amsart}

\usepackage{amssymb} \usepackage{mathtools} \usepackage{graphics}
\usepackage{bm} \usepackage{tikz}
\usetikzlibrary{shapes,snakes,backgrounds}

\usepackage{framed}
\usepackage[boxsize=5mm]{ytableau}
\usepackage{adjustbox}

\newtheorem{theo}{\bf Theorem}[section] \newtheorem{prop}[theo]{\bf
Proposition} \newtheorem{lemma}[theo]{\bf Lemma}
 \newtheorem{defi}[theo]{\bf
Definition} 
 \theoremstyle{remark}
 
\newtheorem{open}[theo]{\bf Open problem}

\DeclareMathOperator{\ess}{ess}

\DeclareMathOperator{\Po}{Po}

\begin{document}

\title[Pomax games -- a family of partizan games played
  on posets]{Pomax games -- a family of integer-valued partizan games played
  on posets}
\author{Erik J\"arleberg \and Jonas Sj{\"o}strand}
\address{Department of Mathematics, Royal Institute of Technology \\
  SE-100 44 Stockholm, Sweden}
\email{erikjar@kth.se}
\email{jonass@kth.se}
\date{Jan 2014}
\keywords{pomax game, poset,
  element-removal game, partizan game, combinatorial game theory,
PSPACE-completeness}
\subjclass[2010]{91A46, 68R05, 06A07}

\begin{abstract}



  We introduce the following class of partizan games, called
  \emph{pomax games}.  Given a partially ordered set whose elements
  are colored black or white, the players Black and White take turns
  removing any maximal element of their own color. If there is no
  such element, the player loses.

  We prove that pomax games are always integer-valued and for colored tree
  posets and chess-colored Young diagram posets
  we give a simple formula for the value of the game.
  However, for pomax games on general posets of height 3 we show
  that the problem of deciding the winner is PSPACE-complete and for posets
  of height 2 we prove NP-hardness.

  Pomax games are just a special case of a larger class of integer-valued games
  that we call \emph{element-removal games}, and
  we pose some open questions regarding
  element-removal games that are not pomax games.
\end{abstract}

\maketitle

\section{Introduction}
A \emph{pomax game} is played as follows.
Given a finite poset $P$ whose elements are colored black or white,
the players Black and White take turns
removing any maximal element of their own color. When
a player cannot make a legal move, he loses the game. As an example,
the pomax game
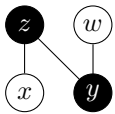
\begin{center}
\begin{tikzpicture}[every path/.style={>=latex},
  wNode/.style={circle, draw=black, inner sep=1pt, minimum size=5mm},
  bNode/.style={circle, draw=white!0!black, fill=black, text=white, outer sep=-1pt, inner sep=1pt, minimum size=5mm}]
  \small \matrix[row sep={9mm,between origins}, column
  sep={9mm,between origins}]{
    \node[bNode] (z) {$z$}; & \node[wNode] (w) {$w$};\\
    \node[wNode] (x) {$x$}; & \node[bNode] (y) {$y$}; \\
  };

  \draw[] (x) edge (z);
  \draw[] (y) edge (z);
  \draw[] (y) edge (w);

\end{tikzpicture}
\end{center}
is a zero game (that is, a second player win):
If Black starts he must remove $z$ and White can counter
by removing $x$, leaving Black with no legal move. If White starts he must
remove $w$, Black must remove $z$, White removes $x$ and finally Black removes
the last element $y$.

With the convention that White is the left (positive) player and Black
is the right (negative) player, one may ask for the game value of a
pomax game in general.  As we will show in
Section~\ref{sec:integervalued}, pomax games are always integer-valued
-- a very rare property among combinatorial games.

Since the birth of modern combinatorial game theory in the 1970s,
hundreds of two-player games with perfect information have been
invented (or discovered) and analyzed.  Most of them are impartial and
thus have nimber values by the Sprague-Grundy Theorem.  Among the
properly partizan games, some are always numbers --
Hackenbush restrained being the most prominent example \cite{conway} -- but, to
the best of our knowledge, essentially only one game studied in the
literature is always integer-valued, namely
Cutcake~\cite[pp.~24--27 and p.~51]{winningways1}.
This game comes in two flavors, Cutcake
and Maundy Cutcake, both of which have a very regular structure that
admits a complete analysis.

Despite being integer-valued, pomax games have a sufficiently rich
structure so that it is PSPACE-complete to decide the winner
of the game, as we will see in Section~\ref{sec:pspacecomplete}.
However, in some special cases the game is computationally tractable,
and in Sections~\ref{sec:balanced} and~\ref{sec:trees} we give
simple formulas for the value of the pomax game played on colored
tree posets and chess-colored Young diagram posets.

\smallskip
Many combinatorial games have been found to be PSPACE-complete,
including common board games like
Checkers, Hex and Reversi \cite{fraenkel78, reisch81, iwatakasai94}
but also more fundamental games like General Geography.
Recently, Grier showed that poset games are PSPACE-complete in general
~\cite{grier}.

A poset game is an impartial game played on a poset, where a legal
move consists of removing any element along with all greater elements.
Examples include the games Nim (where the poset is a disjoint sum
of chains) and Chomp (where the poset is a product of chains).
In a wide sense, pomax games are a partizan variant of poset games, but,
being partizan, they have a quite different role to play
in the abelian group of games.

For a expos\'e over computational complexity results for combinatorial games,
we refer to~\cite{demainehearn}.

\smallskip
Pomax games are just a special case of a larger
class of games that we call \emph{element-removal games}, and when
possible we will state
our results in this more general setting.

The starting position of an element-removal game is a
finite set $X$ whose elements are colored black or white, and in each move
the player (Black or White) removes an element of his own color.
However, not all elements are removable at any stage, but
the set of removable elements is a function of the set $A$ of elements
that are still present. Once an element
becomes removable it may never lose this status until it
is removed. Formally, the removability function
$\rho\colon 2^X\rightarrow2^X$ has the property that
\[
\rho(B)\cap A\subseteq\rho(A)\subseteq A
\]
for any $A\subseteq B\subseteq X$.

Pomax games are the special case where $\rho$ maps $A$ to the
maximal elements of the subposet induced by $A$.

\medskip
The paper is organized as follows.
In Section~\ref{sec:integervalued} we show that element-removal games, and
thus pomax games, are always integer-valued. In Section~\ref{sec:balanced}
we study \emph{balanced} games, a special kind of element-removal games
that are easy to analyze, and in Section~\ref{sec:trees} we give a formula
for the value of any pomax game on a colored tree poset.

After that, we switch our focus to the computational complexity of pomax
games: In
Section~\ref{sec:pspacecomplete} we show that pomax games are
PSPACE-complete even when restricted to height-three posets. As a warm-up,
we show NP-hardness in Section~\ref{sec:nphard}, a result of more than
pedagogical value since it holds already for posets of height two.

Finally, in Section~\ref{sec:future} we suggest some further research
and pose some open questions.

This work was performed at KTH in Stockholm where
the first author was writing his Master's thesis~\cite{jarleberg}, laying the
foundation for this paper.
The second author was supported
by a grant from the Swedish Research Council (621-2009-6090).

\section{Prerequisites}
Here, we will briefly recall those parts of combinatorial game theory
that will be used in the forthcoming sections. No proofs will be given,
but everything follows easily from the comprehensive discussion
in the book ``On Games and Numbers'' by Conway~\cite{conway}.

We will adopt standard notation and terminology for partizan games.
White will always be the left player and Black the right player, and
we will use curly-bracket notation $G=\{G^L\,|\,G^R\}$, where $G^L$ and
$G^R$ are typical left and right options of the game $G$.
The game $\{\,|\,\}$ is called \emph{the zero game} and is denoted by $0$,
and the game $\{\,0\,|\,\}$ is called $1$.

Recall that there is an equivalence relation on games, denoted by an ordinary
equality sign ``$=$'', such that $G=0$ if and only if the second
player wins $G$ (under optimal play). If $G=H$ we will simply say that
\emph{$G$ is equal to $H$}.

The (disjunctive) sum $G+H$ and the negation $-G$ is defined for games, and
the equivalence classes of games form an abelian group under these operations,
with the equivalence class of 0 as zero element.

There are also a partial order on (equivalence classes of) games,
denoted by ``$\ge$'', such that
$G\ge0$ if and only if White wins as a second player.
The order relation is compatible with the group structure.

A game is \emph{integer-valued} if it is equal to a game of
the form $1+1+\dotsb+1$ or its
negation, and the equivalence classes of integer-valued games form
a totally ordered abelian subgroup of the group of all games.

We will use the following sufficient condition for
integer-valueness, which is a simple consequence of the
Simplicity Theorem~\cite[Th.~11]{conway}.
\begin{lemma}\label{lm:integervalueness}
A game is integer-valued if its options are integer-valued and
the difference between any left and right options is at least 2.
In that case, the value of the game is the
integer closest to zero that is strictly larger that any left option and
strictly smaller than any right option.
\end{lemma}

For posets we will write $x\lessdot y$ to denote that $x$
is \emph{covered} by $y$, that is, $x<y$ and there is nothing in between.

\section{Element-removal games are integer-valued}
\label{sec:integervalued}
Clearly, the class of element-removal games (and the class of pomax games)
is closed under
summation and negation, and negation just means
inversion of the coloring so that white elements become black and
vice versa -- it does not affect the removability function.

Our first result is a structure theorem telling us that element-removal
games are very simple objects from an algebraic point of view.
\begin{theo}
Any element-removal game (and thus any pomax game) is integer-valued.
\end{theo}
\begin{proof}
  It suffices to show that, for any element-removal game $G$ and any
  left (White) option $G^L$, we have $G-G^L\ge1$.  By symmetry, this
  will imply that $G-G^R\le-1$ and thus that $G^R-G^L\ge 2$, and by
  Lemma~\ref{lm:integervalueness}, and induction, $G$ will be integer-valued.
  So, if Black starts playing the game
  $G-G^L-1$ we must show that White has a winning strategy.

Let $X$ denote the set of elements of $G$ and let $x\in X$ be the
element that White removed from $G$ to obtain $G^L$.

\begin{description}
\item[Case 1]
Black removes an element $y$ from the $G$ component.
Since $y$ is removable from $X$, it is still removable from
$X\setminus\{x\}$, and thus White may
reply by removing $y$ in the $-G^L$ component. The resulting position
is $G^R-G^{RL}-1$, where $G^R$ is the game obtained from $G$ by removing
the black element $y$ and $G^{RL}$ is obtained from $G^R$ by
removing the white element $x$ (which is removable from $X\setminus\{y\}$
since it is removable from $X$). By (Conway) induction, this game
is nonnegative.

\item[Case 2]
Black removes an element $y$ in the $-G^L$ component.
Then White replies simply by removing $x$ from the $G$ component,
and the resulting position is
$G^L-G^{LL}-1$, where $G^{LL}$ is obtained from $G^L$ by removing $y$.
Again, this is nonnegative by induction.

\item[Case 3]
Black consumes his single move in the $-1$ component.
Then White replies by removing $x$ from the $G$ component, and the
resulting position is $G^L-G^L=0$.
\end{description}

\end{proof}

\section{Balanced games}
\label{sec:balanced}
As we will see in Section~\ref{sec:pspacecomplete}, it is very hard
to compute the value of a pomax game in general (unless
$\text{PSPACE}=\text{P}$).
In this section, however, we will
look at a class of particularly well-behaved element-removal games which we
give the attribute \emph{balanced}. It turns out that the value of
such a game is given simply by the number of
white minus the number of black elements.

\begin{defi}
An element-removal game is \emph{balanced} if it has the following
two properties.
\begin{itemize}
\item
All options are balanced.
\item
If all removable elements are of the same color, then at least half of the
total set of elements have that color.
\end{itemize}
For convenience, we say that a colored poset is balanced if its pomax
game is.
\end{defi}

Thus, a balanced colored poset cannot consist of millions of black elements
covered by a few maximal white elements -- there is always a maximal element
of the majority color.

\begin{prop}\label{pr:balancedissimple}
The value of a balanced game is the number of white elements minus
the number of black elements, and the outcome of the game
is independent of the players' strategies.
\end{prop}
\begin{proof}
Let $G$ be a balanced game with $w$ white elements and $b$ black elements.
Since all options of $G$ are also balanced, by induction,
the value of any left option is
$G^L=w-b-1$ and the value of any right option is
$G^R=w-b+1$.

If $G$ has at least one left option and at least one right
option it follows that $G=w-b$ by Lemma~\ref{lm:integervalueness}.

Suppose $G$ has no right option. Then, since $G$ is balanced, we have
$w\ge b$ and thus $G=\{G^L\,|\,\}=\{w-b-1\,|\,\}=w-b$ by
Lemma~\ref{lm:integervalueness}. The case where
$G$ has no left option is completely analogous.

Since the value of the game is a function of the number of white
and black elements, the outcome does not depend on the
strategies.
\end{proof}

\subsection{Balanced pomax games}
If we are given a poset and want to color it in a way that
will make it balanced, it
seems natural to try a \emph{chess coloring}, namely
a coloring where no element covers an element of the same color.
In this section, we show
that this idea is successful at least for two kinds of posets:
tree posets and Young diagram posets.

In a (non-empty) \emph{tree poset} each element except one -- the root --
covers exactly one element. (For technical reasons, the empty
poset is also considered to be a tree.)
\begin{prop}\label{pr:chesstrees}
The pomax game on a chess-colored tree poset is balanced.
\end{prop}
\begin{proof}
Suppose all maximal elements of the poset are white. Then,
each black element can be paired with one of the white elements
covering it.
\end{proof}

A \emph{Young diagram} (in English notation) is a finite collection of
cells, arranged in left-justified rows, with the row lengths weakly
decreasing.  It can be interpreted as a poset by the rule that a cell
covers the cell immediately to its left and the cell immediately above
it (if those cells exist).\footnote{Young diagram posets can be
  equivalently characterized as being the order ideals of a product of two
  finite chains.}  The maximal cells are called \emph{outer corners}.
Figure~\ref{fig:Youngdiagram} shows an example.
\begin{figure}
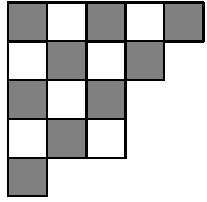

\begin{ytableau}
*(gray) & & *(gray) & & *(gray) \\
& *(gray) & & *(gray) \\
*(gray) & & *(gray) \\
& *(gray) & \\
*(gray)
\end{ytableau}
\caption{A chess-colored Young diagram with four outer
  corners.}\label{fig:Youngdiagram}
\end{figure}

\begin{prop}\label{pr:chessdiagrams}
The pomax game on a chess-colored Young diagram is balanced.
\end{prop}
\begin{proof}
Suppose all outer corners of the Young diagram are white. Then,
each row that ends with a black cell has a row of the same length
immediately below it, and together these two rows have equally many
white as black cells. A row that ends with a white cell has
at least as many white cells as black ones.
\end{proof}

Figure~\ref{fig:counterexamples} shows that
Propositions~\ref{pr:chesstrees} and~\ref{pr:chessdiagrams}
cannot be extended to three-dimensional
(plane partition) diagrams nor to two-dimensional distributive lattices.
However, they can be extended to a larger class of colorings, namely
those avoiding blocking triples.
\begin{figure}
\begin{tikzpicture}[every path/.style={>=latex},
  wNode/.style={circle, draw=black, inner sep=1pt, minimum size=3mm},
  bNode/.style={circle, draw=white!0!black, fill=black, text=white, outer sep=-1pt, inner sep=1pt, minimum size=3mm}]
  \small \matrix[row sep={4mm,between origins}, column
  sep={12mm,between origins}]{
    \node{}; & \node{}; & \node{}; & \node[wNode] (p220) {}; & \node{}; & \node{}; & \node{}; \\
    \node[wNode] (p301) {}; & & & & & & \node[wNode] (p031) {}; \\
    & & & \node[wNode] (p112) {}; & & & \\
    \node[bNode] (p300) {}; & & \node[bNode] (p210) {}; & & \node[bNode] (p120) {}; & & \node[bNode] (p030) {}; \\
    & \node[bNode] (p201) {}; & & \node[bNode] (p111) {}; & & \node[bNode] (p021) {}; & \\
    & & \node[bNode] (p102) {}; & & \node[bNode] (p012) {}; & & \\
    & \node[wNode] (p200) {}; & & \node[wNode] (p110) {}; & & \node[wNode] (p020) {}; & \\
    & & \node[wNode] (p101) {}; & & \node[wNode] (p011) {}; & & \\
    & & & \node[wNode] (p002) {}; & & & \\
    & & \node[bNode] (p100) {}; & & \node[bNode] (p010) {}; & & \\
    & & & \node[bNode] (p001) {}; & & & \\
    & & & & & & \\
    & & & \node[wNode] (p000) {}; & & & \\
  };

  \draw[] (p000) edge (p100);
  \draw[] (p000) edge (p010);
  \draw[] (p000) edge (p001);

  \draw[] (p100) edge (p200);
  \draw[] (p100) edge (p110);
  \draw[] (p100) edge (p101);

  \draw[] (p010) edge (p110);
  \draw[] (p010) edge (p020);
  \draw[] (p010) edge (p011);

  \draw[] (p001) edge (p101);
  \draw[] (p001) edge (p011);
  \draw[] (p001) edge (p002);

  \draw[] (p200) edge (p300);
  \draw[] (p200) edge (p210);
  \draw[] (p200) edge (p201);

  \draw[] (p110) edge (p210);
  \draw[] (p110) edge (p120);
  \draw[] (p110) edge (p111);

  \draw[] (p101) edge (p201);
  \draw[] (p101) edge (p111);
  \draw[] (p101) edge (p102);

  \draw[] (p020) edge (p120);
  \draw[] (p020) edge (p030);
  \draw[] (p020) edge (p021);

  \draw[] (p011) edge (p111);
  \draw[] (p011) edge (p021);
  \draw[] (p011) edge (p012);

  \draw[] (p002) edge (p102);
  \draw[] (p002) edge (p012);

  \draw[] (p300) edge (p301);

  \draw[] (p210) edge (p220);

  \draw[] (p201) edge (p301);

  \draw[] (p120) edge (p220);

  \draw[] (p111) edge (p112);

  \draw[] (p102) edge (p112);

  \draw[] (p030) edge (p031);

  \draw[] (p021) edge (p031);

  \draw[] (p012) edge (p112);
\end{tikzpicture}
\ \ \ \ \ \ \ 
\begin{tikzpicture}[every path/.style={>=latex},
  wNode/.style={circle, draw=black, inner sep=1pt, minimum size=3mm},
  bNode/.style={circle, draw=white!0!black, fill=black, text=white, outer sep=-1pt, inner sep=1pt, minimum size=3mm}]
  \small \matrix[row sep={9mm,between origins}, column
  sep={9mm,between origins}]{
    \node{}; & \node[wNode] (w3) {}; & \node{}; \\
    \node[bNode] (b21) {}; & & \node[bNode] (b22) {}; \\
    & \node[wNode] (w2) {}; & \\
    \node[bNode] (b11) {}; & & \node[bNode] (b12) {}; \\
    & \node[wNode] (w1) {}; & \\
  };

  \draw[] (w1) edge (b11);
  \draw[] (w1) edge (b12);
  
  \draw[] (w2) edge (b11);
  \draw[] (w2) edge (b12);
  \draw[] (w2) edge (b21);
  \draw[] (w2) edge (b22);

  \draw[] (w3) edge (b21);
  \draw[] (w3) edge (b22);

\end{tikzpicture}

\caption{Chess-colored posets with more black
than white elements but with all maximal elements white.
Left: a plane partition poset, that is, an order ideal of a product
of three chains. Right: a
two-dimensional distributive lattice.}\label{fig:counterexamples}
\end{figure}
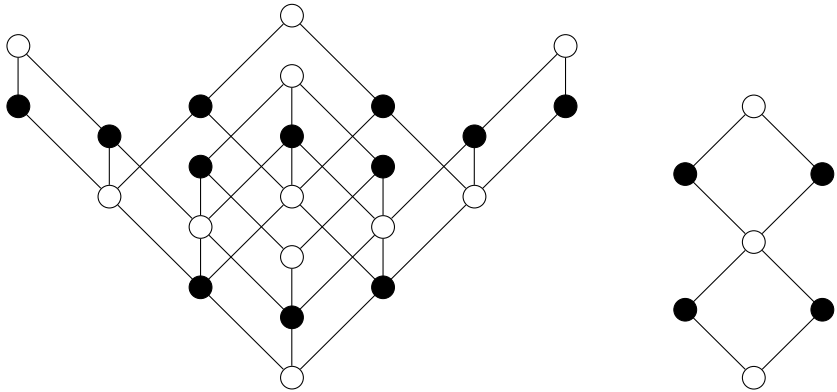

\begin{defi}
A \emph{blocking triple} in a colored poset is a triple of elements
$x\lessdot y\lessdot z$ such that $x$ and $y$ are of the same color and $z$
is of a different color.
\end{defi}

\begin{lemma}\label{lm:noblackneighbors}
Let $P$ be a colored poset without blocking triples and suppose that
all maximal elements are white. Then, no black element is covered
by a black element.
\end{lemma}
\begin{proof}
Let $B$ be the set of black elements that are covered by a black element.
Suppose $B$ is not empty, and let $x$ be an element that is maximal in $B$.
Then $x$ is covered by some black element $y$ not in $B$ which must be
covered by some element $z$ since no black element is maximal in $P$.
Since $y$ does not belong to $B$, the element $z$ must be white,
but this is impossible since $x\lessdot y\lessdot z$ form a blocking triple.
\end{proof}

\begin{prop}\label{pr:balancedtrees}
Any colored tree poset without blocking triples is balanced.
\end{prop}
\begin{proof}
  Suppose all maximal elements are of the same color, white say. Then, by
  Lemma~\ref{lm:noblackneighbors}, each black element is covered by
  some white element.  This pairing shows that there are at least as
  many white as black elements.




\end{proof}

\begin{prop}\label{pr:youngbalanced}
Any colored Young diagram without blocking triples is balanced.
\end{prop}
\begin{proof}
Suppose all outer corners
are of the same color, white say. We want to show that at least
half of the cells are white.

In the light of Lemma~\ref{lm:noblackneighbors} it is easy to see that
any row in the Young diagram has at most one more black cell than white cells,
and this happens only if the row both starts and ends with a black cell.
Furthermore, a row both starting and ending with a white cell has an
excess of white cells.

Any row both starting and ending with a black cell must have a
row immediately below starting and ending with a white cell.





\end{proof}

\subsection{Other balanced element-removal games}
\label{sec:othergames}
The pomax games considered above have several cousins which are
element-removal games but not pomax games. Some of these variants
can be shown to be balanced by the same argument as we used for pomax games.

First, we consider a variant called
\emph{min-max-removal games}. It is an element-removal game played
on a poset, but we let not only the maximal elements but also the
minimal elements be removable.

Starting with a colored tree poset, playing the min-max-removal game
will soon result in a poset consisting of several disjoint trees,
so we ought to formulate our results for such \emph{forest posets}.
The blocking triples turn out to be the right tool also in this situation.
\begin{prop}
The min-max-removal game on any colored forest poset without
blocking triples is balanced.
\end{prop}
\begin{proof}
Identical to the proof of Proposition~\ref{pr:balancedtrees}.
\end{proof}

Starting with a Young diagram poset and playing the min-max-removal game
will soon result in a \emph{skew Young diagram poset}, that is, a
Young diagram with a smaller Young diagram deleted from its
upper-left corner.
\begin{prop}
  The min-max removal game on any colored skew Young diagram poset without
  blocking triples is balanced.
\end{prop}
\begin{proof}
Identical to the proof of Proposition~\ref{pr:youngbalanced}.
\end{proof}

Now, let us throw the whole poset overboard for a while and consider
a couple of element-removal games with a different ground structure.

Given a tree (in the graph-theoretical sense) whose vertices are
colored black or white, the \emph{leaf-removal game} is an element-removal
game on the vertices of the tree,
where the leaves are the removable elements. By a
\emph{chess coloring} we mean a black-white vertex coloring where adjacent
vertices have different colors.
\begin{prop}
The leaf-removal game on any chess-colored tree is balanced.
\end{prop}
\begin{proof}
We can think of our tree as a chess-colored tree poset
by choosing any root vertex (unique minimal element)
and letting all edges (covering relations) be directed from the root.
Then, the proof of Proposition~\ref{pr:chesstrees} applies.
\end{proof}

Finally, let us consider the \emph{corner-removal game},
which is an element-removal game where the ground set
is an $n\times n$ array of colored cells and where a cell is removable if
it is a \emph{corner}, that is, if it has at
most one neighboring cell in the same row and
at most one neighboring cell in the same column. We introduce the term
\emph{truncated square diagrams} for the cell diagrams obtained by
iteratively removing corners from an $n\times n$ cell array.
Figure~\ref{fig:truncatedsquare} shows an example.
\begin{figure}
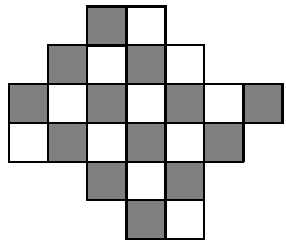

\begin{ytableau}
\none & \none   & *(gray) &         \\
\none & *(gray) &         & *(gray) &      \\
*(gray) &       & *(gray) &         & *(gray) &     & *(gray) \\
        & *(gray) &       & *(gray) &        & *(gray) \\
\none & \none     & *(gray) & & *(gray) \\
\none & \none     & \none & *(gray) & \\
\end{ytableau}
\caption{A chess-colored truncated square diagram with 11
  corners.}\label{fig:truncatedsquare}
\end{figure}

Cells are \emph{neighbors} if they have a common side, and, as always, by a
\emph{chess coloring} we mean a black-white coloring where neighbors have
different colors.

\begin{prop}
  The corner-removal game on any chess-colored truncated square
  diagram is balanced.
\end{prop}
\begin{proof}
Identical to the proof of Proposition~\ref{pr:youngbalanced},
except that there is no need for Lemma~\ref{lm:noblackneighbors}.
\end{proof}

\section{Tree posets}
\label{sec:trees}
In Section~\ref{sec:balanced} we saw that it is easy to compute the
value of the pomax game on a colored tree poset without blocking
triples: Just take the number of white minus the
number of black elements.  In this section we give a complete analysis
of pomax games on tree posets.

Let us begin with a simple example, namely the colored tree poset
in Figure~\ref{fig:chainposet} which is just a chain.
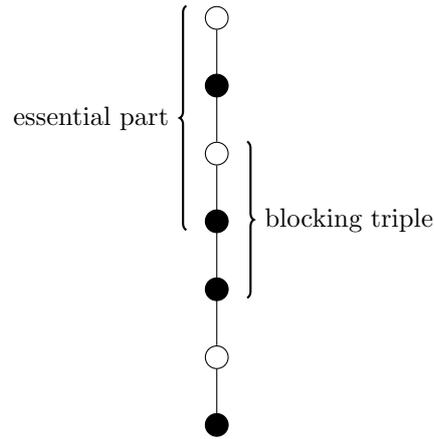
\begin{figure}
\begin{tikzpicture}[every path/.style={>=latex},
  wNode/.style={circle, draw=black, inner sep=1pt, minimum size=3mm},
  bNode/.style={circle, draw=white!0!black, fill=black, text=white, outer sep=-1pt, inner sep=1pt, minimum size=3mm}]
  \small \matrix[row sep={9mm,between origins}, column
  sep={9mm,between origins}]{
    \node[wNode] (p7) {}; \\
    \node[bNode] (p6) {}; \\
    \node[wNode] (p5) {}; \\
    \node[bNode] (p4) {}; \\
    \node[bNode] (p3) {}; \\
    \node[wNode] (p2) {}; \\
    \node[bNode] (p1) {}; \\
  };

  \draw[] (p1) edge (p2);
  \draw[] (p2) edge (p3);
  \draw[] (p3) edge (p4);
  \draw[] (p4) edge (p5);
  \draw[] (p5) edge (p6);
  \draw[] (p6) edge (p7);

  \draw [
  thick,
  decoration={
    brace,
    raise=4mm
  },
  decorate
  ] (p5.north) -- (p3.south)
  node [pos=0.5,anchor=west,xshift=5mm] {blocking triple};

  \draw [
  thick,
  decoration={
    brace,
    mirror,
    raise=4mm
  },
  decorate
  ] (p7.north) -- (p4.south)
  node [pos=0.5,anchor=east,xshift=-5mm] {essential part};

\end{tikzpicture}

\caption{A chain poset with a blocking triple.}\label{fig:chainposet}
\end{figure}
The pomax game on that poset is clearly a zero game: If Black starts
he loses immediately, and if White starts he will lose when the four
topmost elements are removed. Note that the two elements at the bottom
do not affect the value of the game at all. They are ``blocked'' by
the blocking triple above.

Our example suggests the following definition.
\begin{defi}
  For any colored tree poset $P$, its \emph{essential part}, denoted
  by $\ess P$, is the (unique) maximal upper set that does not contain
  any blocking triple.
\end{defi}
We will refer to the elements of the essential part as
\emph{essential} elements.

From now on, we will let $\Po(P)$ denote the pomax game on the colored
poset $P$.

As the following theorem shows, all non-essential elements
might be thrown away without affecting the value
of the game, and since the essential part is balanced its game value
is easy to compute.
\begin{theo}\label{th:ess}
For any colored tree poset $P$, the game equality $\Po(P)=\Po(\ess P)$ holds.
\end{theo}

For the proof we need the following lemma.

\begin{lemma}\label{lm:noopportunity}
  Let $P$ be a black-rooted colored tree poset with at least one white
  element but no blocking triple. Let $m$ be the (integer) game value of
  $\Po(P)$.  Then, in the game $\Po(P)-m$, if Black starts White can
  win before Black gets an opportunity to remove the root of $P$.
\end{lemma}
\begin{proof}
By Propositions~\ref{pr:balancedtrees} and~\ref{pr:balancedissimple},
White will win $\Po(P)-m$ when Black starts, no matter what strategies they
use. If White removes all white elements in the $-m$ component (if $m$ is
negative) before making any move in the $\Po(P)$ component, she will never
have to remove all white elements in the $\Po(P)$ component, and thus
the root will never be removable for Black.
\end{proof}

\begin{proof}[Proof of Theorem~\ref{th:ess}]
We assume that $\ess P\ne P$; otherwise there is nothing to prove.

The essential part consists of a disjoint
union of trees $\ess P=T_1\cup T_2\cup \dotsb\cup T_k$ and the non-essential
part $P\setminus\ess P$ is a tree. For $i=1,\dotsc,k$, let
$m_i$ be the value of $T_i$
(which is just the number of white minus the number of black elements
since $T_i$ does not contain any blocking triple). We want to show
that the game
$\Po(P)-m_1-m_2-\dotsb-m_k$ is a win for the second player. By symmetry,
it suffices to show that White will win if Black starts.

Note that, by construction of the essential part and by our assumption
that $\ess P\ne P$,
none of the trees $T_1,\dotsc,T_k$ is unicolored.
Thus, by Lemma~\ref{lm:noopportunity}, if Black starts White can win
$\Po(\ess P)-m_1-m_2-\dotsb-m_k$ without ever giving Black an opportunity
to remove a minimal element of $\ess P$. By adopting this strategy
to the game $\Po(P)-m_1-m_2-\dotsb-m_k$, 
White can win without removing any non-essential element.
Black will not get the chance to remove any non-essential element, because
each black maximal element $x$ of $P\setminus\ess P$ is covered by some black
minimal element of $\ess P$ -- otherwise $x$ would have been essential.
\end{proof}

\section{Pomax games of height 2 are NP-hard}
\label{sec:nphard}
Up to this point all our results have been about the \emph{simplicity}
of pomax games: They are integer-valued and their values are easy to compute
in some cases, in particular if the poset is a tree. In this and
the forthcoming section, however, we will show that in general
it is very hard to find the winner of a pomax game, even for very shallow
posets. (All this is under the assumption that $\text{PSPACE}\ne\text{P}$.)

By the \emph{height} of a poset we mean the length of its longest chain.

\begin{theo}\label{th:nphard}
The problem of deciding whether a given pomax game equals zero is NP-hard
even if the height of the colored poset is restricted to two.
\end{theo}

\begin{proof}
Recall that a Boolean formula is on \emph{conjunctive normal form (CNF)}
if it is a conjunction of clauses, where
each clause is a disjunction of literals, each literal being a variable
or the negation of a variable. If every clause has exactly three literals,
it is a \emph{3CNF-formula}. An example is
$(x_1\vee\neg x_2\vee\neg x_4)\wedge(x_2\vee\neg x_3\vee x_4)$.

We will make a reduction from the canonical NP-complete problem 3-SAT.
\begin{center}
\begin{framed}
{\bf 3-Satisfiability (3-SAT)}
\begin{description}
\item[Input]
A 3CNF-formula.
\item[Output] ``Yes'' if and only if the formula is true for some assignments
of the variables.
\end{description}
\end{framed}
\end{center}

Given a 3CNF-formula we will construct a colored poset (in
polynomial time) whose pomax game is zero precisely if the formula is
true.

For each variable $x_i$ in the formula we put two white \emph{assignment
elements} in the poset, one called ``$x_i=0$'' and one called
``$x_i=1$'' (where $0$ and $1$ should be interpreted as ``false'' and
``true'', respectively).  Also, for each clause $C_j$ in the formula
we put a black \emph{clause element} $c_j$ in the poset and we let it be
covered by exactly those assignment elements that would make the
clause false. For instance, the clause element corresponding to
$x_1\vee x_2\vee\neg x_3$ would be covered by the assignment
elements ``$x_1=0$'', ``$x_2=0$'' and ``$x_3=1$''.

We want that the removal of an assignment element ``$x_i=\alpha$''
during play should correspond to actually assigning the value $\alpha$
to the variable $x_i$, so we need some mechanism to prevent White
from cheating by removing both ``$x_i=0$'' and ``$x_i=1$''. This is
accomplished by letting ``$x_i=0$'' and ``$x_i=1$'' cover a black \emph{candy
element} so that White cannot cheat without
uncovering candy for his opponent.  

Finally, we put as many black isolated elements in the poset as there
are Boolean variables, so that Black has something to eat while White
is trying to satisfy the formula.  Figure~\ref{fig:poset_np} shows an
example of our construction.

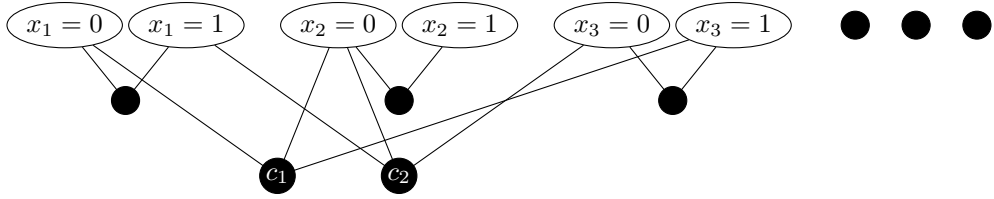
\begin{figure}
  \begin{tikzpicture}[every path/.style={>=latex},
    wEllipseNode/.style={ellipse, draw=black, inner sep=1pt, minimum height=6mm},
    bEllipseNode/.style={ellipse, draw=black, fill=black, text=white, inner sep=1pt, minimum height=6mm},
    wNode/.style={circle, draw=black, inner sep=1pt, minimum size=3.7mm},
    bNode/.style={circle, draw=white!0!black, fill=black, text=white, outer sep=-1pt, inner sep=1pt, minimum size=3.7mm}]
    \small
    \matrix[row sep={10mm,between origins},column sep={4mm,between origins}]{
      \node{}; & \node{}; & \node[wEllipseNode] (x10) {$x_1=0$}; & \node{}; & \node{}; & \node{}; & \node[wEllipseNode] (x11) {$x_1=1$}; & \node{}; & \node{}; &
      \node{}; & \node{}; & \node[wEllipseNode] (x20) {$x_2=0$}; & \node{}; & \node{}; & \node{}; & \node[wEllipseNode] (x21) {$x_2=1$}; & \node{}; & \node{}; &
      \node{}; & \node{}; & \node[wEllipseNode] (x30) {$x_3=0$}; & \node{}; & \node{}; & \node{}; & \node[wEllipseNode] (x31) {$x_3=1$}; & \node{}; & \node{}; &
      \node{}; & \node[bNode]{}; & \node{}; & \node[bNode]{}; & \node{}; & \node[bNode]{}; \\
      & & & & \node[bNode] (d1) {}; & & & & &
      & & & & \node[bNode] (d2) {}; & & & & &
      & & & & \node[bNode] (d3) {}; \\
      & & & & & & & & & \node[bNode] (c1) {$c_1$};
      & & & & \node[bNode] (c2) {$c_2$}; \\
    };
    \draw[] (x10) edge (c1);
    \draw[] (x20) edge (c1);
    \draw[] (x31) edge (c1);
    
    \draw[] (x11) edge (c2);
    \draw[] (x20) edge (c2);
    \draw[] (x30) edge (c2);
    
    \draw[] (x10) edge (d1);
    \draw[] (x11) edge (d1);
    \draw[] (x20) edge (d2);
    \draw[] (x21) edge (d2);
    \draw[] (x30) edge (d3);
    \draw[] (x31) edge (d3);
  \end{tikzpicture}
  \caption{The colored poset
    constructed from the 3CNF-formula
    $(x_1\vee x_2\vee\neg x_3)\wedge(\neg x_1\vee x_2\vee x_3)$.}
  \label{fig:poset_np}
\end{figure}

If White starts he cannot win, because Black has an isolated element
for each pair of white assignment elements, and if White cheats Black gets
candy.

If Black starts, White will win unless some of the black clause elements
are uncovered during the game. Clearly,
White can avoid uncovering a clause element
precisely if the 3CNF-formula is satisfiable.
\end{proof}

\section{Pomax games of height 3 are PSPACE-complete}
\label{sec:pspacecomplete}
Since the number of moves during a pomax game is bounded by the
size of the poset, its outcome can be determined by an algorithm
using only a polynomial amount of space. In this section we show
that pomax games are in fact PSPACE-complete.

\begin{theo}\label{th:pspacecomplete}
The problem of deciding whether a given pomax game equals zero is
PSPACE-complete
even if the height of the colored poset is restricted to three.
\end{theo}
\begin{proof}
We will make a reduction from the following archetypical
PSPACE-complete problem.
\begin{center}
\begin{framed}
{\bf Quantified boolean formula problem (QBF)}
\begin{description}
\item[Input]
A QBF-formula, that is, a formula of the type
\[
\forall x_1\exists x_2\forall x_3\exists x_4\dotsm\forall
x_{n-1}\exists x_n\phi(x_1,\dotsc,x_n),
\]
where $\phi$ is a CNF-formula $C_1\wedge C_2\wedge\dotsm\wedge C_m$.
The number $n$ of variables is even.
\item[Output] ``Yes'' if and only if the QBF-formula is true.
\end{description}
\end{framed}
\end{center}
We will think of QBF as the problem of deciding the winner of a
two-player game where the players, let us call them Black and White,
assign truth values to the variables $x_i$. Black assigns variables
with odd indices and White assigns variables with even indices.
Furthermore, Black must assign $x_1$ first and then White, with
knowledge of the value of $x_1$, must assign $x_2$, and so on. When
all $n$ variables have been assigned, White wins if the CNF-formula
$\phi$ becomes true.

Given a QBF-formula as above we will construct a colored poset (in
polynomial time) whose pomax game is zero precisely if the formula is
true. Let us build this poset step by step, initially focusing on the
main picture and taking care of the details as we go along.

Like in the proof of Theorem~\ref{th:nphard},
for each variable $x_i$ in the formula we put two \emph{assignment
elements} in the poset, one called ``$x_i=0$'' and one called
``$x_i=1$''. But now we color the elements black if $i$ is odd
and white if $i$ is even.

Again following the proof of Theorem~\ref{th:nphard},
for each clause $C_j$ in the formula
we put a black \emph{clause element} $c_j$ in the poset and we let it be
covered by exactly those assignment elements that would make the
clause false.

As before, we need some mechanism to prevent players
from cheating by removing both ``$x_i=0$'' and ``$x_i=1$''. This is
accomplished by letting ``$x_i=0$'' and ``$x_i=1$'' cover some \emph{candy
elements} of the opposite color so that a player cannot cheat without
uncovering lots of candy for his opponent.  From now on we assume that
there is enough candy to make sure that no player will ever cheat.
(Obviously, if cheating uncovers more candy elements than the total
number of non-candy elements in the poset, there will be no cheating.
A more careful analysis shows that it suffices to have $m+1$ white
candy elements for each variable with odd index and one single black
candy element for each variable with even index.)
Figure~\ref{fig:poset_half} shows an example of a colored poset as
constructed so far.

\begin{figure}
\begin{adjustbox}{center}
\begin{tikzpicture}[every path/.style={>=latex},
  wEllipseNode/.style={ellipse, draw=black, inner sep=1pt, minimum height=6mm},
  bEllipseNode/.style={ellipse, draw=black, fill=black, text=white, inner sep=1pt, minimum height=6mm},
  wNode/.style={circle, draw=black, inner sep=1pt, minimum size=3.7mm},
  bNode/.style={circle, draw=white!0!black, fill=black, text=white, outer sep=-1pt, inner sep=1pt, minimum size=3.7mm}]
\small
  \matrix[row sep={10mm,between origins}, column sep={4mm,between origins}]{
    \node{}; & \node{}; & \node[bEllipseNode] (x10) {{$x_1=0$}}; & \node{}; & \node{}; & \node{}; & \node[bEllipseNode] (x11) {{$x_1=1$}}; & \node{}; & \node{}; & \node{}; &
    \node{}; & \node{}; & \node[wEllipseNode] (x20) {{$x_2=0$}}; & \node{}; & \node{}; & \node{}; & \node[wEllipseNode] (x21) {{$x_2=1$}}; & \node{}; & \node{}; & \node{}; &
    \node{}; & \node{}; & \node[bEllipseNode] (x30) {{$x_3=0$}}; & \node{}; & \node{}; & \node{}; & \node[bEllipseNode] (x31) {{$x_3=1$}}; & \node{}; & \node{}; & \node{}; &
    \node{}; & \node{}; & \node[wEllipseNode] (x40) {{$x_4=0$}}; & \node{}; & \node{}; & \node{}; & \node[wEllipseNode] (x41) {{$x_4=1$}}; \\
    & & & \node[wNode] (d11) {}; & \node{$\,\cdots$}; & \node[wNode] (d13) {}; & & & & &
    & & & \node[bNode] (d21) {}; & \node{$\,\cdots$}; & \node[bNode] (d23) {}; & & & & &
    & & & \node[wNode] (d31) {}; & \node{$\,\cdots$}; & \node[wNode] (d33) {}; & & & & &
    & & & \node[bNode] (d41) {}; & \node{$\,\cdots$}; & \node[bNode] (d43) {}; \\
    & & & & & & & & & & & & & & & & & & & \node[bNode] (c1) {$c_1$}; &
    & & & & & & & & & \node[bNode] (c2) {$c_2$}; \\
};

  \draw[] (x10) edge (c1);
  \draw[] (x21) edge (c1);
  \draw[] (x41) edge (c1);

  \draw[] (x20) edge (c2);
  \draw[] (x31) edge (c2);
  \draw[] (x40) edge (c2);

  \draw[] (x10) edge (d11);
  \draw[] (x11) edge (d11);
  \draw[] (x10) edge (d13);
  \draw[] (x11) edge (d13);
  
  \draw [
  thick,
  decoration={
    brace,
    mirror,
    raise=4mm
  },
  decorate
  ] (d11.west) -- (d13.east)
  node [pos=0.5,anchor=north,yshift=-5mm] {candy};

  \draw[] (x20) edge (d21);
  \draw[] (x21) edge (d21);
  \draw[] (x20) edge (d23);
  \draw[] (x21) edge (d23);

  \draw[] (x30) edge (d31);
  \draw[] (x31) edge (d31);
  \draw[] (x30) edge (d33);
  \draw[] (x31) edge (d33);

  \draw[] (x40) edge (d41);
  \draw[] (x41) edge (d41);
  \draw[] (x40) edge (d43);
  \draw[] (x41) edge (d43);

\end{tikzpicture}
\end{adjustbox}

\caption{The assignment, clause and candy elements of the colored poset
  constructed from the QBF instance $\forall x_1\exists x_2\forall
  x_3\exists x_4\,(x_1\vee\neg x_2\vee\neg x_4)\wedge(x_2\vee\neg x_3\vee x_4)$.}
\label{fig:poset_half}
\end{figure}
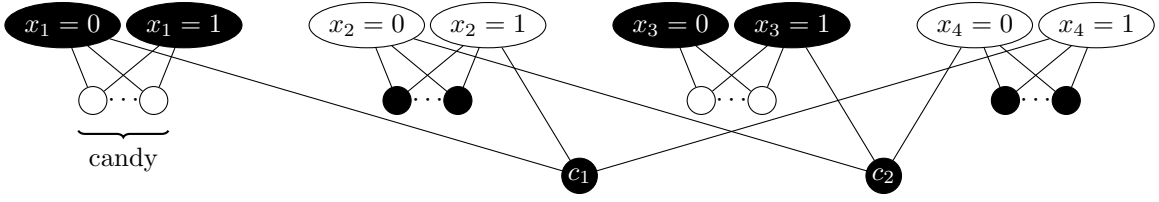

The idea is that Black would start the game and assign a value to
$x_1$ by choosing to remove either ``$x_1=0$'' or ``$x_1=1$''. Then,
White would remove either ``$x_2=0$'' or ``$x_2=1$'' and Black would
remove either ``$x_3=0$'' or ``$x_3=1$'' and so on. Finally, White
would remove either ``$x_n=0$'' or ``$x_n=1$'' and she will win the
game if no clause element $c_j$ has been uncovered, which is the case
exactly if the CNF-formula $\phi$ is true. However, nothing in the
present construction will force the players to make the assignments
in the correct order from left to right.

For each $i\in\{1,\dotsc, n-1\}$, to make sure that the player making
the assignment of the variable $x_{i+1}$ will not have to do that
before the other player has assigned the previous variable $x_i$, we
install a gadget consisting of six new elements called $a_i^0$,
$a_i^1$, $b_i^{00}$, $b_i^{01}$, $b_i^{10}$, $b_i^{11}$, and the
covering relations $\text{``$x_{i+1}=\beta$''}\gtrdot a_i^{\beta}\gtrdot
b_i^{\alpha\beta}$ and $\text{``$x_i=\alpha$''}\gtrdot b_i^{\alpha\beta}$ for
$\alpha,\beta\in\{0,1\}$.  We color $a_i^\beta$ black if $i$ is odd
and white if $i$ is even, and $b_i^{\alpha\beta}$ white if $i$ is odd
and black if $i$ is even.

This completes our construction, and the result is exemplified in
Figure~\ref{fig:poset_complete}.

\begin{figure}
\begin{adjustbox}{center}
\begin{tikzpicture}[every path/.style={>=latex},
  wEllipseNode/.style={ellipse, draw=black, inner sep=1pt, minimum height=6mm},
  bEllipseNode/.style={ellipse, draw=black, fill=black, text=white, inner sep=1pt, minimum height=6mm},
  wNode/.style={circle, draw=black, inner sep=1pt, minimum size=3.7mm},
  bNode/.style={circle, draw=white!0!black, fill=black, text=white, outer sep=-1pt, inner sep=1pt, minimum size=3.7mm}]
\small
  \matrix[row sep={10mm,between origins}, column sep={4mm,between origins}]{
    \node{}; & \node{}; & \node[bEllipseNode] (x10) {{$x_1=0$}}; & \node{}; & \node{}; & \node{}; & \node[bEllipseNode] (x11) {{$x_1=1$}}; & \node{}; & \node{}; & \node{}; &
    \node{}; & \node{}; & \node[wEllipseNode] (x20) {{$x_2=0$}}; & \node{}; & \node{}; & \node{}; & \node[wEllipseNode] (x21) {{$x_2=1$}}; & \node{}; & \node{}; & \node{}; &
    \node{}; & \node{}; & \node[bEllipseNode] (x30) {{$x_3=0$}}; & \node{}; & \node{}; & \node{}; & \node[bEllipseNode] (x31) {{$x_3=1$}}; & \node{}; & \node{}; & \node{}; &
    \node{}; & \node{}; & \node[wEllipseNode] (x40) {{$x_4=0$}}; & \node{}; & \node{}; & \node{}; & \node[wEllipseNode] (x41) {{$x_4=1$}}; \\
    & & & \node[wNode] (d11) {}; & \node{$\,\cdots$}; & \node[wNode] (d13) {}; & & & & &
    & & & \node[bNode] (d21) {}; & \node{$\,\cdots$}; & \node[bNode] (d23) {}; & & & & &
    & & & \node[wNode] (d31) {}; & \node{$\,\cdots$}; & \node[wNode] (d33) {}; & & & & &
    & & & \node[bNode] (d41) {}; & \node{$\,\cdots$}; & \node[bNode] (d43) {}; \\
    & & & \node[bNode] (a10) {$a_1^0$}; & & \node[bNode] (a11) {$a_1^1$}; & & & & &
    & & & \node[wNode] (a20) {$a_2^0$}; & & \node[wNode] (a21) {$a_2^1$}; & & & & \node[bNode] (c1) {$c_1$}; &
    & & & \node[bNode] (a30) {$a_3^0$}; & & \node[bNode] (a31) {$a_3^1$}; & & & & \node[bNode] (c2) {$c_2$}; \\
    \node[wNode] (b100) {$b_1^{00}$}; & & \node[wNode] (b101) {$b_1^{01}$}; & & & & \node[wNode] (b110) {$b_1^{10}$}; & & \node[wNode] (b111) {$b_1^{11}$}; & &
    \node[bNode] (b200) {$b_2^{00}$}; & & \node[bNode] (b201) {$b_2^{01}$}; & & & & \node[bNode] (b210) {$b_2^{10}$}; & & \node[bNode] (b211) {$b_2^{11}$}; & &
    \node[wNode] (b300) {$b_3^{00}$}; & & \node[wNode] (b301) {$b_3^{01}$}; & & & & \node[wNode] (b310) {$b_3^{10}$}; & & \node[wNode] (b311) {$b_3^{11}$}; \\
};

  \draw[] (x10) edge (c1);
  \draw[] (x21) edge (c1);
  \draw[] (x41) edge (c1);

  \draw[] (x20) edge (c2);
  \draw[] (x31) edge (c2);
  \draw[] (x40) edge (c2);

  \draw[] (x10) edge (d11);
  \draw[] (x11) edge (d11);
  \draw[] (x10) edge (d13);
  \draw[] (x11) edge (d13);

  \draw[] (x20) edge (d21);
  \draw[] (x21) edge (d21);
  \draw[] (x20) edge (d23);
  \draw[] (x21) edge (d23);

  \draw[] (x30) edge (d31);
  \draw[] (x31) edge (d31);
  \draw[] (x30) edge (d33);
  \draw[] (x31) edge (d33);

  \draw[] (x40) edge (d41);
  \draw[] (x41) edge (d41);
  \draw[] (x40) edge (d43);
  \draw[] (x41) edge (d43);

  \draw[] (x10) edge (b100);
  \draw[] (x10) edge (b101);
  \draw[] (x11) edge (b110);
  \draw[] (x11) edge (b111);
  \draw[] (a10) edge (b100);
  \draw[] (a10) edge (b110);
  \draw[] (a11) edge (b101);
  \draw[] (a11) edge (b111);
  \draw[] (x20) edge (a10);
  \draw[] (x21) edge (a11);

  \draw[] (x20) edge (b200);
  \draw[] (x20) edge (b201);
  \draw[] (x21) edge (b210);
  \draw[] (x21) edge (b211);
  \draw[] (a20) edge (b200);
  \draw[] (a20) edge (b210);
  \draw[] (a21) edge (b201);
  \draw[] (a21) edge (b211);
  \draw[] (x30) edge (a20);
  \draw[] (x31) edge (a21);

  \draw[] (x30) edge (b300);
  \draw[] (x30) edge (b301);
  \draw[] (x31) edge (b310);
  \draw[] (x31) edge (b311);
  \draw[] (a30) edge (b300);
  \draw[] (a30) edge (b310);
  \draw[] (a31) edge (b301);
  \draw[] (a31) edge (b311);
  \draw[] (x40) edge (a30);
  \draw[] (x41) edge (a31);

\end{tikzpicture}
\end{adjustbox}
\caption{The colored poset constructed from the QBF instance $\forall
  x_1\exists x_2\forall x_3\exists x_4\,(x_1\vee\neg x_2\vee\neg
  x_4)\wedge(x_2\vee\neg x_3\vee x_4)$.}
\label{fig:poset_complete}
\end{figure}
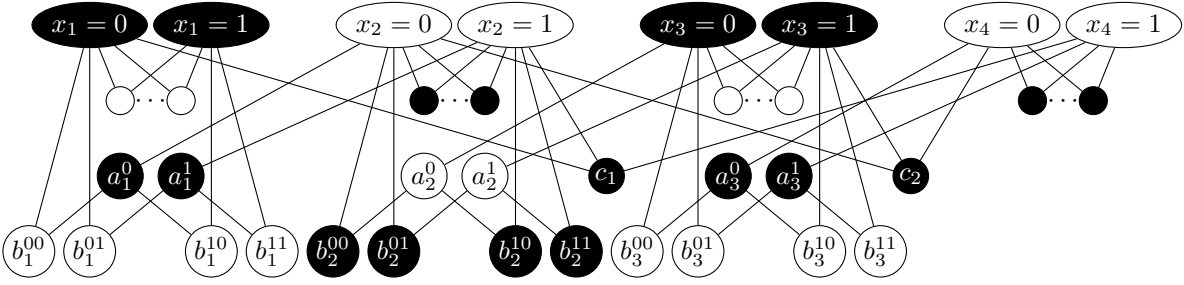

Note that, since no player cheats and uncovers candy for the opponent, for
each $i$, any time during play at most one of the elements $a_i^0$
and $a_i^1$ is maximal and at most one of the elements
$b_i^{00}$, $b_i^{01}$, $b_i^{10}$ and $b_i^{11}$ is maximal.

A player, let us say White, does not gain anything from removing
``$x_{i+1}=\beta$'' while the previous pair of assignment elements
``$x_i=0$'' and ``$x_i=1$'' are both still present, because the other
player, Black, could answer immediately by removing the element
$a_i^\beta$ without uncovering any white element.  Not until later
when Black removes ``$x_i=\alpha$'' for some $\alpha\in\{0,1\}$, White
is compensated by the uncovering of the white element
$b_i^{\alpha\beta}$, so White could as well have waited for this to
happen before she removed ``$x_{i+1}=\beta$''.

We conclude that, if Black starts the game, White will win,
and hence the game is $\ge0$, if and only
if the QBF-formula is true. If White starts the game, Black will win by
simply removing $a_i^\beta$ whenever White removes ``$x_{i+1}=\beta$'', so
the game is always $\le0$.
\end{proof}

\section{Future research and open questions}
\label{sec:future}
Theorems~\ref{th:nphard} and~\ref{th:pspacecomplete} leave us with an
obvious open question.
\begin{open}
Is it a PSPACE-complete problem to compute the outcome of a given pomax game
even if the height of the colored poset is restricted to two?
\end{open}
Colored posets like the one in Figure~\ref{fig:poset_half} seem
very hard to analyze, and though the players may cheat by assigning
the variables in the wrong order, we would guess that games of this type
are PSPACE-complete.
There is also a theorem by
Schaefer~\cite[Th.~3.8]{schaefer78} that points in this direction.

The posets constructed in the proofs of Theorems~\ref{th:nphard}
and~\ref{th:pspacecomplete} have small height but they might be quite
high-dimensional. One could ask if it is possible
to trade low height for low dimensionality while still maintaining the
hardness of the problem.
\begin{open}
How computationally hard is the problem of computing the outcome of a pomax
game on a colored Young diagram poset?
\end{open}

In Section~\ref{sec:othergames} we defined some particular element-removal
games that are not pomax games, and we saw that they behave well
if their underlying structure (poset, tree graph or cell diagram)
is chess-colored. In particular min-max-removal games on forest posets
and leaf-removal games might be possible to analyze for any coloring
by essentially the same method we used for pomax games on tree posets in
Section~\ref{sec:trees}.

\begin{open}
Find a formula for the value of the min-max-removal game on any colored
forest poset.
\end{open}

\begin{open}
Find a formula for the value of the leaf-removal game on any colored
tree.
\end{open}

As mentioned in the introduction, pomax games are a partizan variant
of poset games.  But there is a more straightforward way to
make a poset game partizan and that is simply to color the elements
and let the player at turn choose any element of his own color and
remove it along with all greater elements (even if some of those happen to be of
the opposite color). The games
so obtained, let us call them \emph{partizan poset games}, seem to be
related to Hackenbush restrained. For instance, it is easy to see that
they equal numbers (by essentially the same argument as for
Hackenbush restrained, see~\cite[p.~87]{conway}), and every restrained
Hackenbush tree is obviously equivalent to a partizan poset game on a
colored tree poset. This latter observation shows that partizan poset games
are not
integers but can take the value of any dyadic rational number.
However, the similarity with restrained Hackenbush apparently disappears
for more complex posets (or more complex Hackenbush graphs).

We think that a more thorough study of partizan poset games would be worthwhile.

\bibliographystyle{amsplain}
\bibliography{pomax}

\end{document}